\documentclass[10pt]{amsart}
\usepackage{amsfonts,amsmath,amsthm,amssymb}
\usepackage{caption}
\usepackage{url}
\usepackage{enumerate}
\usepackage{etoolbox}
\usepackage{hyperref}

\makeatletter
\patchcmd{\@startsection}
  {\@afterindenttrue}
  {\@afterindentfalse}
  {}{}
\makeatother

\newtheorem{theorem}{Theorem}[section]
\newtheorem{corollary}[theorem]{Corollary}

\newtheorem{lemma}[theorem]{Lemma}
\newtheorem{example}[theorem]{Example}

\newtheorem{proposition}[theorem]{Proposition}

\theoremstyle{definition}
\numberwithin{equation}{section}

\theoremstyle{remark}
\newtheorem{rem}[theorem]{Remark}
\theoremstyle{remark}

\newcommand{\beql}[1]{\begin{equation}\label{#1}}
\newcommand{\eeq}{\end{equation}}

\begin{document}

\title[Equiangular lines in Euclidean spaces]{Equiangular lines in Euclidean spaces}

\author[G. Greaves, J. H. Koolen, A. Munemasa, and F. Sz\"oll\H{o}si]{Gary Greaves, Jacobus H. Koolen,\\ Akihiro Munemasa, and Ferenc Sz\"oll\H{o}si}

\date{\today.}

\address{G. G., A. M., and F. Sz.: Research Center for Pure and Applied Mathematics, Graduate School of Information Sciences,
Tohoku University, Sendai 980-8579, Japan}\email{grwgrvs@gmail.com, munemasa@math.is.tohoku.ac.jp, szoferi@gmail.com}\address{J. H. K.: School of Mathematical Sciences, University of Science and Technology of China, Hefei, Anhui, 230026, P. R. China}\email{koolen@ustc.edu.cn}



\begin{abstract}
We obtain several new results contributing to the theory of real equiangular line systems. Among other things, we present a new general lower bound on the maximum number of equiangular lines in $d$ dimensional Euclidean space; we describe the two-graphs on $12$ vertices; and we investigate Seidel matrices with exactly three distinct eigenvalues. As a result, we improve on two long-standing upper bounds regarding the maximum number of equiangular lines in dimensions $d=14$, and $d=16$. Additionally, we prove the nonexistence of certain regular graphs with four eigenvalues, and correct some tables from the literature.
\end{abstract}
\subjclass[2010]{Primary 05B20, secondary 05B40}
\maketitle

	
{\bf Keywords and phrases.} {\it Equiangular lines, Seidel matrix, Switching, Two-graph}

\section{Introduction}
Let $d\geq 1$ be an integer and let $\mathbb{R}^d$ denote the Euclidean $d$-dimensional space equipped with the usual inner product $\left\langle .,.\right\rangle$. A set of $n\geq 1$ lines, represented by the unit vectors $v_1, v_2, \hdots, v_n\in\mathbb{R}^d$, is called \emph{equiangular} if there exists a constant $\alpha\geq0$, such that \mbox{$\left\langle v_i,v_j\right\rangle=\pm\alpha$} for all $1\leq i<j\leq n$. This constant is referred to as the \emph{common angle} between the lines. If $\alpha=0$ then the line system is just a subset of an orthonormal basis. If $n\leq d$ then it is easy to construct equiangular line systems for all $0\leq \alpha\leq 1$. Hence we exclude these trivial cases by assuming that $n>d$ and consequently $\alpha>0$.

Equiangular lines were introduced first by Haantjes \cite{HAA} in $1948$ and then investigated by Van Lint and Seidel in a seminal paper \cite{vLS} and were further studied during the $1970$s \cite{BMS, C, N, TTH}. Recently there has been an interest in some special, highly structured \emph{complex} equiangular line systems \cite{GR}. In the engineering literature these objects are called \emph{tight frames}, and they are used for various applications in signal processing \cite{FS, HP}. Some of these complex tight frames are studied by physicists under the name of \emph{SIC-POVMs}, and are used in quantum tomography \cite{ABG, SG}. Discussion of the complex case and these applications are beyond the scope of this paper.

This paper is motivated by the fundamental problem of determining the \emph{maximum number}, $N(d)$, \emph{of equiangular lines} in $\mathbb{R}^d$. In Table~\ref{xyz} we display the current best known lower and upper bounds on $N(d)$ for the first few values of $d$. The table includes the contributions of this paper; compare to the relevant table in \cite{LS}.

\[{\arraycolsep=1.3pt
\begin{array}{c|cccccccccccccccc}
d        & 2 & 3        & 4           & 5  & 6  & 7$--$13 & 14  & 15 & 16 & 17 & 18 & 19 & 20 & 21  & 22  & 23$--$41\\  
N(d)        & 3 & 6        & 6           & 10 & 16 & 28    & 28$--$29  & 36 & 40$--$41 & 48$--$50 & 48$--$61 & 72$--$76 & 90$--$96 & 126 & 176 & 276\\
1/\alpha & 2 & \sqrt{5} & \sqrt{5}, 3 & 3  & 3  & 3     & 3,5 & 5  & 5  & 5  & 5  & 5  & 5  & 5   & 5   & 5\\
\end{array}}\]\label{xyz}
\begingroup
\captionof{table}{The maximum number of equiangular lines for $d\leq 41$.}%
\endgroup
We remark that there exist several incorrectly revised tables in the current literature (see Remark~\ref{absnewrem}), which might suggest to the uninitiated that $N(d)$ is known for small $d$. Table~\ref{xyz} shows that, despite a considerable amount of research in the past $40$ years, determining $N(d)$ even for relatively small values of $d$ is still out of reach. The methods used to obtain configurations with the above indicated number of lines are discussed throughout the scattered literature \cite{LS, HIG, T, TTH, JCT}, while the upper bounds will be mentioned later. The reader may wish to jump ahead to Section~\ref{subs3}, where the new results regarding dimensions $d=14$ and $d=16$ are presented.
\begin{rem}\label{absnewrem}
Seidel seems to claim in \cite[Section 3.3]{HIG} that the lower bounds indicated in Table~\ref{xyz} above cannot be improved unless $d=19$ or $20$. This might be true, but it is unclear whether or not his statement follows implicitly from the cited literature.
\end{rem}
The Gram matrix of the equiangular line system $\left[G\right]_{i,j}:=\left\langle v_i,v_j\right\rangle$, $1\leq i,j\leq n$, is of fundamental interest. It contains all the relevant parameters of the line system and thus the study of equiangular lines using matrix theoretical and linear algebraic tools is possible. We find it, however, more convenient to consider the \emph{Seidel matrix} $S:=(G-I)/\alpha$ instead, which is a symmetric matrix with zero diagonal and $\pm1$ entries otherwise. The multiplicity of the smallest eigenvalue $\lambda_{0}=-1/\alpha$ of $S$ describes the smallest possible dimension $d$ such that the line system can be embedded into $\mathbb{R}^d$ (see \cite{LS} for more details). Seidel matrices are the central objects of this paper.

The outline of this paper is as follows. In Section $2$ we provide an overview of the subject and present both classical and recent results on fundamental properties of equiangular lines. The main result of this part is an improved general quadratic lower bound on the maximum number of equiangular lines (cf.\ \cite{DC}). In Section~$3$ we prove new modular identities involving the determinant and the permanent of Seidel matrices. In Section~$4$ we report on our computational results: we present a full classification of Seidel matrices up to order $12$ (cf.\ \cite{BMS, MK, TS}). In Section~$5$ we investigate Seidel matrices with three distinct eigenvalues, and we give new upper bounds for $N(d)$ for several values of $d$. In Section~$6$ we discuss equiangular line systems with common angle $1/5$. During the course of the paper we extend (and occasionally correct) several tables from the literature.

\section{Large sets of equiangular lines}
In this section we recall some fundamental structural results on equiangular lines, and discuss several lower and upper bounds on $N(d)$. The main contribution here is Corollary~\ref{newthm}, which extends an earlier construction of de Caen (see the discussion preceding Theorem~\ref{larges}) and describes a new general lower bound for $N(d)$.

The first result is the \emph{absolute bound}.
\begin{theorem}[Gerzon, see \cite{LS}]\label{absbd}
For every $d\geq 2$ we have
\beql{absolute}
N(d)\leq\frac{d(d+1)}{2},
\eeq
and if equality holds then $d=2$, $3$, or $d+2$ is the square of an odd integer.
\end{theorem}
Equality indeed holds for $d=2$, $3$, $7$, and $23$. For long it was conjectured that the upper bound \eqref{absolute} is attained whenever $d+2$ is an odd square, until Makhnev, in a surprising breakthrough paper \cite{MAK}, proved that $N(47)<24\cdot47$. Additional counterexamples were found subsequently by a different method \cite{BMV}; and recently, by employing semidefinite programming, the upper bound \eqref{absolute} for $N(d)$ was considerably improved for various dimensions $d\leq136$ \cite{ABARG}. It is worthwhile to note that these are three fundamentally different approaches.

It turns out that the common angle $\alpha$ is heavily restricted when $n>2d$.
\begin{theorem}[Neumann, see \cite{LS}]\label{s2t2}
Assume that there exist $n$ equiangular lines in $\mathbb{R}^d$ with common angle $\alpha$.
\begin{enumerate}[a$)$]
\item If the associated Seidel matrix has an even eigenvalue, then it has multiplicity $1$.
\item If $n>2d$, then $1/\alpha$ is an odd integer.
\end{enumerate}
\end{theorem}
Theorem~\ref{s2t2} is particularly useful for deciding the maximum number of equiangular lines in $\mathbb{R}^d$ for small $d$. The result above can be combined with the following improved upper bound, called the \emph{relative bound}, which takes into consideration the common angle under some assumptions on the ambient dimension $d$.
\begin{proposition}[cf.~\mbox{\cite[Lemma~6.1]{vLS}}]\label{s2p23}
Assume that there exist $n$ equiangular lines in $\mathbb{R}^d$ with common angle $\alpha\leq1/\sqrt{d+2}$. Then $n\leq d(1-\alpha^2)/(1-d\alpha^2)$, and equality holds if and only if the corresponding Seidel matrix has exactly two distinct eigenvalues.
\end{proposition}
\begin{rem}
The relative bound, described by Proposition~\ref{s2p23}, does not exceed the absolute bound, given by Theorem~\ref{absbd}.
\end{rem} 
We provide a slight generalization of Proposition~\ref{s2p23} in Section~\ref{XYZZ} (see Theorem~\ref{3g}), and investigate the case of ``almost'' equality in Theorem~\ref{maintnon}. For other bounds, similar in spirit, see \cite{OKUDA}.

Now we turn to the discussion of lower bounds. It is easy to show that $N(d)\geq cd\sqrt{d}$ for some $c>0$ (see \cite{LS}), but for many years the asymptotic behavior of $N(d)$ was unknown (see \cite[p.~ 884]{HIG}). De Caen \cite{DC} was the first to show that $N(d)\geq 2(d+1)^2/9$ for infinitely many values of $d$. His results imply a general quadratic lower bound, with a smaller leading coefficient, namely $N(d)\geq (d+2)^2/72$ for every $d\geq 1$, see \cite[pp.~79--80]{FSZ}. We analyze de Caen's construction and improve upon the previously mentioned general lower bound by a factor of $2$.
\begin{theorem}\label{larges}
For each $i\geq 1$ and $m=4^i$ there exists an equiangular set of
\begin{enumerate}[$($a$)$]
\item $m(m/2+1)$ lines in $\mathbb{R}^{3m/2+1}$$;$ and
\item $m(m/2+1)-1$ lines in $\mathbb{R}^{3m/2}$$;$ and
\item $mj$ lines in $\mathbb{R}^{m+j-1}$, for every $j$ satisfying $1\leq j\leq m/2$.
\end{enumerate}
In each of the three cases the common angle is $\alpha=1/(2^i+1)$.
\end{theorem}
The proof is based on the existence of a complete set of $m/2+1$ pairwise \emph{mutually unbiased bases} in $\mathbb{R}^m$, whenever $m=4^i$ for some $i\geq 1$. For a gentle introduction see \cite{MATE}; for more details see Kantor's work\footnote{Note that the quadratic form $x_1x_3+x_1x_4+x_2x_3+x_2x_4$ should be replaced by $x_1x_3+x_1x_4+x_2x_3+x_3x_4$ in the relevant example of \cite[p.~157]{WK}.} \cite[p.~157]{WK}, and \cite{CCKS, WKxd}. We remark here that the case $j=m/2$ in part~(c) gives exactly the same number of lines in $\mathbb{R}^{3m/2-1}$ as de Caen's construction \cite{DC}.
\begin{proof}
Let $i\geq 1$ and consider a complete set of real mutually unbiased bases in $\mathbb{R}^m$, where $m=4^i$, i.e., $m/2+1$ orthonormal bases represented by orthogonal matrices $B_0:=I$, $\hdots$, $B_{m/2}$. On the one hand, every entry of $B_j^TB_k$ is $\pm1/\sqrt{m}$, where $0\leq j<k\leq m/2$. On the other hand, we may assume that the first row of each of the matrices $B_j$, $1\leq j\leq m/2$ is normalized to be $e_m/\sqrt{m}$, where $e_j$ denotes a row vector of length $j$ having all entries $1$.  Having these matrices at our disposal, we build a block matrix whose columns will form the desired configuration of equiangular lines. Let $a=1/\sqrt{\sqrt{m}+1}$ and $b=a\sqrt[4]{m}$. Consider
\[M(j):=\left[\begin{array}{ccccc}
bI & bB_1 & \dots & bB_{j-1} & bB_{j}\\
&&aI_{j+1}\otimes e_m &&
\end{array}\right],\quad 1\leq j\leq m/2.\]
Part (a) follows from $M(m/2)$. Part (b) follows by discarding the first column of $M(m/2)$, and observing that $\left[\begin{array}{ccc}-b&0\otimes e_{m}& a\otimes e_{m/2}\end{array}\right]$ is in its left nullspace. Therefore the dimension of the column space is at most $3m/2$. Part (c) follows by discarding the first $m$ columns of $M(j)$, and observing that $\left[\begin{array}{ccc}-b & 0\otimes e_{m}& a\otimes e_{j}\end{array}\right]$ and $\left[\begin{array}{ccc} 0\otimes e_{m}& 1 & 0\otimes e_{j}\end{array}\right]$ are in its left nullspace.
\end{proof}
\begin{rem}
To maximize the number of equiangular lines in a given dimension $d$ the choice of the angle $\alpha$ is crucial. Indeed, in $\mathbb{R}^{95}$ we have $n\leq 438$ for $\alpha\leq 1/11$ by Proposition~\ref{s2p23}. However, for $\alpha=1/9$, for which the previously mentioned bound does not apply, we can construct $2048$ lines via part~(c) of Theorem~\ref{larges}. On the other hand, for $\alpha=1/3$ the maximum number of equiangular lines is just $188$ (see \cite{LS}). Therefore, in general, using an angle that is too small or too large does not give rise to a maximum number of equiangular lines.
\end{rem}
We note some consequences.
\begin{corollary}\label{newthm}
Let $d\geq 25$ and let $m=4^i$ with $i\geq 2$ be the unique power of $4$ for which $3m/2+1\leq d\leq 6m$. Then
\[N(d)\geq\begin{cases}
m(m/2+1)\qquad &  3m/2+1\leq d\leq 33m/8-1;\\
4m(d-4m+1)\qquad\text{for}\qquad & 33m/8\leq d\leq 6m-1;\\
4m(2m+1)-1\qquad & d=6m.\\
\end{cases}\]
\end{corollary}
\begin{proof}
The result is an immediate consequence of Theorem~\ref{larges}. The first case follows from part~(a) for $i\geq 2$. The second case follows from part~(c) by replacing $m$ with $4m$. Indeed, by using $j$ mutually unbiased bases of size $4m$, where $m/8+1\leq j\leq 2m$, we have $4m j\geq 4m(m/8+1)>m(m/2+1)$ lines in dimension $d=4m+j-1\geq 33m/8$. The third case follows from part~(b) by replacing $m$ with $4m$.
\end{proof}
We remark that de Caen constructs equiangular lines for $d=6m-1$, ($m=4^i$ for some $i\geq 0$), and his construction gives the same lower bound as indicated by Corollary~\ref{newthm}.
\begin{corollary}[cf.\ \cite{DC} and \mbox{\cite[pp.~79--80]{FSZ}}]\label{s2t1}
For every $d\geq2$ we have
\[N(d)\geq \left\lceil\frac{1}{1089}\left(32d^2+328d+296\right)\right\rceil.\]
\end{corollary}
\begin{proof}
Let $f(d):=(32d^2+328d+296)/1089$. For $2\leq d\leq 22$ the statement is trivial since $f(d)<d$. For $23\leq d\leq 24$ we have $f(d)<276$ and hence the $276$ lines in $\mathbb{R}^{23}$ (see Table \ref{xyz}) satisfy the claim. Therefore we can assume that $d\geq 25$ and so invoke Corollary~\ref{newthm}. We readily see that $f$ is the quadratic curve fit to the set of points $\{(33m/8-1,m(m/2+1))\colon m=4^i, i\geq 2\}\subset\mathbb{R}\times\mathbb{R}$. By noting that $N(33m/8)\geq 4m(m/8+1)>f(33m/8)$ for $m=4^i$, $i\geq 2$, the convexity of $f$ yields the desired result.
\end{proof}
It would be interesting to see whether the gap between the constants $32/1089$ and $1/2$ (cf.\ Theorem~\ref{absbd}) could be further reduced. This could be achieved by either obtaining better constructions resulting in larger lower bounds on $N(d)$; or by showing the non-existence of certain configurations, obtaining smaller upper bounds on $N(d)$.
\section{Some structural results on Seidel matrices}
The \emph{switching class} of a Seidel matrix $S$ is the set of all Seidel matrices $PDSDP^T$ where $P$ is a permutation matrix and $D$ is a $\pm1$ diagonal matrix. Two Seidel matrices, $S_1$ and $S_2$, are called \emph{switching equivalent}, if $S_2=PDS_1DP^T$ holds for some $P$ and $D$. This equivalence captures the symmetries of the equiangular line systems. Indeed, the described operations correspond to relabeling the spanning unit vectors and replacing some of them with their negatives. It is customary to associate an \emph{ambient graph} $\Gamma$ to a Seidel matrix $S$, defined via its adjacency matrix $A:=(J-S-I)/2$, where $J$ is the matrix with all entries $1$. In this terminology, by switching with respect to a vertex $v$ in $\Gamma$, we obtain the graph $\Gamma'$ in which the neighbors of $v$ in $\Gamma'$ are its nonneighbors in $\Gamma$ and \emph{vice versa}.

In this section we state some new structural results on Seidel matrices. To motivate our efforts we begin this part with a result about the determinant of a Seidel matrix whose ambient graph is regular.
\begin{lemma}\label{triww2}
Let $\Gamma$ be a connected $k$-regular graph on $n$ vertices with adjacency matrix $A$. Assume that $A$ has exactly $r$ distinct eigenvalues. Then $S:=J-2A-I$ has at most $r$ distinct eigenvalues.
\end{lemma}
\begin{proof}
If the eigenvalues of $A$ are $k$, and $\lambda_1<\dots<\lambda_{r-1}$, then the eigenvalues of $S$ are $n-2k-1$, and $-2\lambda_i-1$ for all $1\leq i\leq r-1$.
\end{proof}
\begin{lemma}\label{S3L2}
Let $\Gamma$ be a connected $k$-regular graph on $n$ vertices with $e$ edges and let $A$ be its adjacency matrix. Then $\mathrm{det}\left(J-2A-I\right)\equiv (-1)^n(1-n)+4en\ (\mathrm{mod}\ 8)$.
\end{lemma}
In particular, if the switching class of a Seidel matrix contains a connected regular graph then we know its determinant modulo $8$.
\begin{proof}
The proof follows from a standard spectral analysis. Let $k$ and $\lambda_2,\lambda_3,\hdots,\lambda_n$ be the eigenvalues of $A$. Note that $\mathrm{tr}A=0$ and $\mathrm{tr}A^2=nk=2e$. The eigenvalues of $S:=J-2A-I$ are $n-2k-1$ and $-2\lambda_i-1$ for $2\leq i\leq n$ from the proof of Lemma \ref{triww2}. We find
\begin{align*}
\mathrm{det}S&=(n-2k-1)\prod_{i=2}^n(-2\lambda_i-1)\\
&\equiv(-1)^n(2k+1-n)\left(1+2\sum_{i=2}^n\lambda_i+4\sum_{2\leq i<j\leq n}\lambda_i\lambda_j\right)\\
&=(-1)^n(2k+1-n)(1-2k+4(k^2-e))\\
&=(-1)^n(1-n+8k^3-16ek+4en)\equiv(-1)^n(1-n)+4en\ (\mathrm{mod}\ 8).\qedhere
\end{align*}
\end{proof}
We generalize Lemma~\ref{S3L2} and prove it for arbitrary graphs (see Theorem~\ref{S2mTH}), along with a related result on permanents. Let us recall that a \emph{derangement} is a permutation without any fixed points. It is clear from the definition that the number of derangements on $n$ elements, $\delta(n)$, is exactly $\mathrm{per}(J-I)$. By expanding this permanent along the first row we find that $\delta(n)=(n-1)(\delta(n-1)+\delta(n-2))$ for $n\geq 3$. Combining this with the initial values $\delta(1)=0$ and $\delta(2)=1$ and using induction on $n$ we obtain for $n\geq 2$ that
\beql{S2df}\delta(n)=n\delta(n-1)+(-1)^n.\eeq
\begin{rem}\label{S2Rem}
It is possible to determine $\delta(n)\ (\mathrm{mod}\ 8)$ based on equation~\eqref{S2df}. An easy inductive argument shows that $\delta(n)\equiv 1\ (\mathrm{mod}\ 8)$ for $n$ even, and $\delta(n)\equiv n-1\ (\mathrm{mod}\ 8)$ for $n$ odd. In particular $\delta(n)\equiv n-1\ (\mathrm{mod}\ 2)$.
\end{rem}
\begin{lemma}\label{S3mL}
Let $\left[S\right]_{i,j}=s_{i,j}$ be a Seidel matrix of order $n\geq2$. Let $S'$ be obtained from $S$ by changing $s_{12}=s_{21}$ to their negatives. Then, we have $\mathrm{det}S'\equiv\mathrm{det}S+4n\ (\mathrm{mod}\ 8)$.
\end{lemma}
We write $\mathfrak{S}_n$ to denote the symmetric group on $n$ elements.
\begin{proof}
For $n=2$ the result is immediate, therefore we can assume that $n\geq 3$. Let us define $R_n:=\{\sigma\in \mathfrak{S}_n\colon\sigma(1)=2,\sigma(2)\neq 1,\sigma(i)\neq i\text{ for }\ 3\leq i\leq n\}$ and note that its cardinality is $\delta(n-1)$. Since the inversion number obeys $I(\sigma)=I(\sigma^{-1})$ we readily find that
\[\mathrm{det}S-\mathrm{det}S'=4\sum_{\substack{\sigma\in R_n}}(-1)^{I(\sigma)}\prod_{i=1}^ns_{i,\sigma(i)},\]
and hence, since we are working modulo $8$, it is enough to determine the parity of the above sum, or equivalently, since all terms are $\pm1$, the parity of $\delta(n-1)$. By Remark~\ref{S2Rem} this is exactly the same as the parity of $n$.
\end{proof}
The following is the main contribution of this section.
\begin{theorem}[cf. \mbox{\cite[p.~ 659]{HAM}}]\label{S2mTH}
Let $\Gamma$ be a graph on $n\geq 2$ vertices with $e$ edges and let $A$ be its adjacency matrix. Then $\mathrm{det}(J-2A-I)\equiv(-1)^n(1-n)+4en\ (\mathrm{mod}\ 8)$.
\end{theorem}
\begin{proof}
Let $S=J-I$ and observe that $S'=J-2A-I$ can be obtained from $S$ by changing the sign of $2e$ off-diagonal entries. Since $\mathrm{det}S=(-1)^{n-1}(n-1)$, repeated application of Lemma~\ref{S3mL} yields the desired result.
\end{proof}
We will use Theorem~\ref{S2mTH} in the following qualitative form.
\begin{corollary}\label{S2mC}
Let $S$ be a Seidel matrix of order $n$. Then $\mathrm{det}S\equiv 1-n\ (\mathrm{mod}\ 4)$.
\end{corollary}
\begin{proof}
If $n=1$ then $\mathrm{det}S=0$. Otherwise the statement follows from Theorem~\ref{S2mTH}.
\end{proof}
We point out the following analogous property.
\begin{theorem}\label{S2mC2}
Let $\Gamma$ be a graph on $n\geq2$ vertices with $e$ edges and let $A$ be its adjacency matrix. Then $\mathrm{per}(J-2A-I)\equiv (-1)^n+n(1-(-1)^n)/2+4en\ (\mathrm{mod}\ 8)$.
\end{theorem}
\begin{proof}
Follows along the same lines as the proof of Lemma~\ref{S3mL} and Theorem~\ref{S2mTH}, \emph{mutatis mutandis}. Since $\mathrm{per}(J-I)=\delta(n)$, we find that $\mathrm{per}(J-2A-I)\equiv\delta(n)+4en\ (\mathrm{mod}\ 8)$. The result follows from Remark~\ref{S2Rem}.
\end{proof}
As an application of Theorem~\ref{S2mTH}, we prove the following result. Recall that a Seidel matrix $S$ is \emph{self-complementary} if $S$ and $-S$ are switching equivalent \cite{NEWSCREF}.
\begin{proposition}\label{pselfc1}
Let $S$ be a Seidel matrix of order $n\equiv 3\ (\mathrm{mod}\ 4)$. Then $\mathrm{det}(-S)\equiv\mathrm{det}S+4\ (\mathrm{mod}\ 8)$. In particular, $S$ cannot be self-complementary.
\end{proposition}
\begin{proof}
From Theorem~\ref{S2mTH} it follows that $\mathrm{det}S\equiv\pm2\ (\mathrm{mod}\ 8)$ and hence $\mathrm{det}(-S)=-\mathrm{det}S\equiv\mp2\equiv\pm2+4\equiv\mathrm{det}S+4\ (\mathrm{mod}\ 8)$.
\end{proof}
We do not know any applications of Theorem~\ref{S2mC2}, but we will use Corollary~\ref{S2mC} several times in the following sections. Understanding the determinant modulo further values might lead to some non-existence results on Seidel matrices with prescribed spectrum.
\section{Small sets of equiangular lines}\label{sec5}
In this section we report on some computational results on Seidel matrices. Prior to this work one representative from each of the switching classes was available up to order $n\leq 11$, see \cite{BMS, MK, TS}. The aim of this section is to present a complete classification of Seidel matrices up to order $n\leq12$. We achieve this result in the following way: first, we present a list of Seidel matrices of order $n=12$ thus complementing the earlier results mentioned above; second, and perhaps more importantly, we provide an efficient way to determine the equivalence of Seidel matrices of these orders using various invariants.

Interestingly, the number of switching classes are known explicitly, due to the following spectacular equicardinality result. Recall that an \emph{Euler graph} is a (not necessarily connected) graph, all of whose vertices are of even degree.
\begin{theorem}[Mallows--Sloane \cite{MS}]\label{MTMS}
The number of switching classes of Seidel matrices of order $n$ and the number of Euler graphs on $n$ vertices are equal in number.
\end{theorem}
Since the number of Euler graphs are known due to explicit formulae \cite{L, R} (see also \cite{C, MS}), once enough inequivalent switching classes are identified, they constitute a full classification. It is worthwhile to note that while, for $n$ odd, each switching class contains a unique Euler graph \cite{Sx, SEL}, this property fails to hold for $n$ even \cite{MS}.

Let $\varphi$ be a map defined on the set of Seidel matrices of order $n$. We call this map a \emph{complete invariant} if $\varphi(S_1)$=$\varphi(S_2)$ holds if and only if the Seidel matrices $S_1$ and $S_2$ are equivalent. A carefully chosen complete invariant could provide a convenient way to determine equivalence of Seidel matrices. Therefore the most time-consuming part of the classification, which is equivalence rejection, can be avoided by employing suitably chosen invariants. The choice of the invariants introduced in this section was motivated by the graph reconstruction conjecture (see \cite{BMR}).

Let $\mathrm{det}(xI-S)=x^n+\sum_{i=0}^{n-2} a_ix^i$ be the characteristic polynomial of a Seidel matrix $S$ of order $n$. Given an explicit list of Seidel matrices (see \cite{BMS, TS}) it is easy to check that the truncated coefficient list, denoted by
\beql{S3coefL}
\chi_n(S):=\left[a_0,a_1,\hdots,a_{n-2}\right]
\eeq
is a complete invariant for $n\leq 7$. One slight inconvenience is the data structure of $\chi_n$: comparing and storing vectors is infeasible for higher $n$ due to the large number of inequivalent matrices. Instead, we map various vector valued invariants into the cyclic group $\mathbb{Z}/m\mathbb{Z}$, which we identify with $\{0,1,2,\hdots,m-1\}$.
\begin{example}\label{S5ML}
Let $S$ be a Seidel matrix of order $n$. If $n\leq 6$ then $\chi_n(S)$, defined in \eqref{S3coefL}, is a complete invariant. For $n\in\{7,8,9,10\}$ the following functions $\varphi_n\colon \mathcal{M}_{n}(\mathbb{Z})\to\mathbb{Z}/m\mathbb{Z}$ defined as
\begin{align*}\varphi_7(S)&:=\prod\limits_{a_i\in\chi_7(S)}a_i\ (\mathrm{mod}\ 409), & \varphi_8(S)&:=\prod\limits_{S'}\varphi_7(S')\ (\mathrm{mod}\ 7507),\\
\varphi_9(S)&:=\prod\limits_{S'}\varphi_8(S')\ (\mathrm{mod}\ 268921), & \varphi_{10}(S)&:=\prod\limits_{S'}\varphi_9(S')\ (\mathrm{mod}\ 45131767),
\end{align*}
where $S'$ runs through all $(n-1)\times (n-1)$ principal minors, are complete invariants.
\end{example}
The invariants $\varphi_{8}, \varphi_{9}$ and $\varphi_{10}$ in Example~\ref{S5ML} are recursively defined, and their computation becomes less and less efficient. Therefore we devise an improved invariant for $n=11$, one which can be computed faster, as follows. Let $S$ be a Seidel matrix of order $11$, and let us define the following functions:
\begin{gather}\psi(S):=\{(v_i,m_i)\colon\text{$v_i$ is a $9\times 9$ principal minor of $S$ of multiplicity $m_i$}\}\nonumber,\\
\varphi_{11}(S):=\mathrm{det}S\prod\limits_{(v_i,m_i)\in\psi(S)}(v_i+1)(m_i+1)\ (\mathrm{mod}\ 97124414801).\label{phi11}\end{gather}
Formula \eqref{phi11} was found by trial and improvement. We have the following.
\begin{proposition}\label{S5MP}
The Seidel matrices $S_1$ and $S_2$ of order $11$ are switching equivalent if and only if $\varphi_{11}(S_1)=\varphi_{11}(S_2)$, where $\varphi_{11}$ is defined in \eqref{phi11}.
\end{proposition}
\begin{proof}
This claim can be verified by using the classification of Euler graphs on $11$ vertices \cite{MK}.
\end{proof}
The main achievement of this section is the following result.
\begin{theorem}\label{S5MT}
The Seidel matrices $S_1$ and $S_2$ of order $12$ are switching equivalent if and only if $\varphi_{12}(S_1)=\varphi_{12}(S_2)$, where $\varphi_{11}$ is defined in \eqref{phi11} and $\varphi_{12}(S)$ is the following multiset$:$
\[\varphi_{12}(S)=\{\varphi_{11}(S')\colon\text{$S'$ is a $11\times 11$ principal submatrix of $S$}\}.\]
\end{theorem}
\begin{proof}
From Theorem~\ref{MTMS} we know that there exist exactly $87723296$ switching classes of order $12$ (see an explicit table in \cite{MS}). We have generated this number of matrices with distinct values of $\varphi_{12}$, which are available as a supplement on the web page \cite{WEB}.
\end{proof}
As an application, we recall that if $S$ is a Seidel matrix of order $n$ with eigenvalues $\lambda_i$, $1\leq i\leq n$, then $\mathcal{S}(S):=\sum_{i=1}^n|\lambda_i|$ is the \emph{energy} of $S$. Haemers conjectures, and  R.\ Swinkels verifies for $n\leq 10$, that $\mathcal{S}(S)\geq 2(n-1)$ for all $S$ (see \cite{RND, HAM, KM}). We contribute some additional empirical evidence.
\begin{corollary}
Let $S$ be a Seidel matrix of order $n\leq 12$. Then $\mathcal{S}(S)\geq 2(n-1)$, with equality if and only if $S=\pm(J_n-I_n)$, up to switching equivalence.
\end{corollary}
\begin{proof}
We have verified this by computing the spectrum in {\ttfamily Mathematica} numerically.
\end{proof}
Recall that the \emph{automorphism group} of a Seidel matrix $S$, denoted by $\mathrm{Aut}(S)$, is formed by the pairs $(P,D)$ where $P$ is a permutation matrix, $D$ is a diagonal $\pm1$ matrix, such that $S=PDSDP^T$, and where we do not make a distinction between the group elements $(P,D)$ and $(P,-D)$. Having classified all Seidel matrices of order $12$, we have enumerated those having an interesting automorphism group. One of the (many) fascinating properties of Seidel matrices is that their automorphism group can be larger than the automorphism group of any of the ambient graphs contained in their switching class. This property is captured by Cameron's \emph{first cohomology invariant} $\gamma$ \cite{C}, which is nonzero if and only if $|\mathrm{Aut}(S)|>\max\{|\mathrm{Aut}(\Gamma)|\colon \text{$\Gamma$ is an ambient graph in the switching class of $S$}\}$. Since almost all two-graphs satisfy $\gamma=0$ \cite{C}, those for which this does not hold are of exceptional interest. The number of such Seidel matrices up to order $n\leq 12$ is presented in the concluding summarizing table (cf.\ \cite{BMS, C}, but note that both papers incorrectly report the number of $8\times 8$ Seidel matrices with $\gamma\neq 0$).
\[{\arraycolsep=3pt
\begin{array}{c|cccccccccccc}
n   & 1       & 2 			& 3 			& 			4 & 			5 & 			6 & 	 	 7  & 8             & 9            & 10    & 11 & 12\\
\hline
\text{total}  & 1       & 1 			& 2 			& 			3 & 			7 & 		 16 & 		 54 & 243           & 2038         & 33120 & 1182004 & 87723296\\
\gamma\neq0     & 0 & 0 & 0 & 0 & 0 & 2  & 0  & 21  & 0    & 392   & 0 & 15274\\
\text{self-compl.}        & 1 & 1 & 0 & 1 & 1 & 4 & 0 & 19 & 10 & 360 & 0 & 25112\\
\lambda_{\min}=-5 & 0 & 0 & 0 & 0 & 0 & 1 & 2 & 8 & 33 & 306 & 6727 & 219754\\
\text{invariant} & \mathrm{det} & \mathrm{det} & \mathrm{det} & \chi_4 & \chi_5 & \chi_6 & \varphi_7 & \varphi_8 & \varphi_9 & \varphi_{10} & \varphi_{11} & \varphi_{12}
\end{array}}\]
\begingroup
\captionof{table}{Summary of Seidel matrices of order $n\leq 12$.}
\endgroup
The number of self-complementary Seidel matrices is known explicitly \cite{NEWSCREF}. Note that by Proposition~\ref{pselfc1} there do not exist such Seidel matrices of order $n\equiv 3\ (\mathrm{mod}\ 4)$. Also $\gamma\neq 0$ can hold only if $n$ is even \cite{C}.
\begin{rem}\label{remuselater}
We remark here that instead of considering Seidel matrices, one might want to construct the frame vectors $v_i$, $1\leq i\leq n$ directly. This can be achieved by recursively extending a linearly independent set of frame vectors. It is not too hard to see that, up to change of basis, each coordinate of the frame vectors squares to a rational number.
\end{rem}
\section{Two-graphs with three eigenvalues}\label{XYZZ}
\emph{Two-graphs} were introduced as a tool for studying doubly transitive groups \cite{T}. We do not recall the concept here rigorously, but rather remark that two-graphs form a class of $3$-uniform hypergraphs and are in one-to-one correspondence with switching classes of Seidel matrices (see \cite[p.~ 881]{HIG}). A two-graph is said to be \emph{regular} if the corresponding Seidel matrix has exactly two distinct eigenvalues. These objects are extremely useful as they correspond to the equality case in the relative bound (see Proposition~\ref{s2p23}). However, for many values of $n$, there is no regular two-graph on $n$ points. We observe that there exist various large sets of equiangular line systems, whose corresponding Seidel matrices have exactly three distinct eigenvalues (see Examples \ref{rank3g}, \ref{ex415}, \ref{STS19}, and \ref{ASCh}) and therefore we begin to investigate them in this section. The analogous question for graph adjacency matrices is considered in a series of papers \cite{GRW, CO, VD, MzK}.%
%
\subsection{Existence and structure}\label{subs1}
In this section we completely settle the existence of Seidel matrices with exactly three distinct eigenvalues (see Theorem~\ref{vacuous}). We also show in Theorem~\ref{Euler3graph} that every such Seidel matrix is switching equivalent to one having an Euler graph as an ambient graph.

We begin by recalling a fundamental tool, crucial to this section, called \emph{interlacing}.
\begin{lemma}[see \cite{HB}]\label{intL}
Let $\left[\begin{smallmatrix} A & B\\ B^T & C\end{smallmatrix}\right]$ be a real symmetric matrix of order $n$ with eigenvalues $\lambda_1\leq\dots\leq \lambda_n$ where $A$ is of order $m$ with eigenvalues $\mu_1\leq\dots\leq\mu_m$. Then we have $\lambda_i\leq \mu_i\leq \lambda_{n-m+i}$ for all $i$ with $1\leq i\leq m$.
\end{lemma} 
In what follows we will be interested in Seidel matrices of order $n$, having spectrum $\{\left[\lambda\right]^{n-d},\left[\mu\right]^m,\left[\nu\right]^{d-m}\}$, where $2\leq d\leq n-1$ and $1\leq m\leq d-1$. If $\lambda$ happens to be the smallest eigenvalue then such a matrix corresponds to $n$ equiangular lines in $\mathbb{R}^d$. We state what we will refer to as the \emph{standard equations} which we will frequently employ during the course of this section:
\begin{align}
\mathrm{tr}S&=(n-d)\lambda+m\mu+(d-m)\nu=0,\label{St1}\\
\mathrm{tr}S^2&=(n-d)\lambda^2+m\mu^2+(d-m)\nu^2=n(n-1)\label{St2}.
\end{align}
Since $S$ is real and symmetric, its minimal polynomial is $(x-\lambda)(x-\mu)(x-\nu)$, we obtain
\beql{CHT}S^3=(\lambda+\mu+\nu)S^2-(\lambda\mu+\mu\nu+\nu\lambda)S+\lambda\mu\nu I.\eeq
We get another formula by cubing $S=J-2A-I$, namely
\begin{equation}\label{CHTv2}
\begin{split}
S^3&=(n^2-3n+3)J+4(A^2J+JA^2)+4AJA-2JAJ\\
&+2(3-n)(AJ+JA)-8A^3-12A^2-6A-I.
\end{split}
\end{equation}
We start with a structural result, interesting in its own right.
\begin{theorem}\label{Euler3graph}
Let $S$ be a Seidel matrix with exactly three distinct eigenvalues. Then $S$ is switching equivalent to a Seidel matrix $S'$, whose ambient graph is an Euler graph.
\end{theorem}
\begin{proof}
Assume that $S$ is of order $n$. For $n$ odd, the statement follows from a general result due to Seidel \cite{Sx, SEL}. Hence, we assume that $n$ is even. Moreover, up to switching equivalence, we can assume that $S$ is a Seidel matrix whose ambient graph $\Gamma$ has $e$ edges and has a vertex $v_1$ with vertex degree $\mathrm{d}(v_1)=0$. From equation \eqref{CHT} we observe that for all $i$ with $1\leq i\leq n$ we have
\beql{uj1}[S^3]_{ii}=[S^3]_{11}.\eeq
From equation \eqref{CHTv2} it follows, by using $JAJ=2eJ$, that
\beql{uj2}[S^3]_{ii}=n^2-3n+2-4e+4[A^2J+JA^2]_{ii}-8[A^3]_{ii}+4(\mathrm{d}(v_i))^2-4n\mathrm{d}(v_i).\eeq
Plug in $i=1$ to conclude that $[S^3]_{11}=n^2-3n+2-4e$. Therefore, by combining equations~\eqref{uj1} and \eqref{uj2}, it follows that $4[A^2J+JA^2]_{ii}-8[A^3]_{ii}+4(\mathrm{d}(v_i))^2-4n\mathrm{d}(v_i)=0$ for all $i$ with $1\leq i\leq n$. Since $n$ is even so is $(\mathrm{d}(v_i))^2$ for all $i$ with $1\leq i\leq n$, and hence $\Gamma$ is an Euler graph as claimed.
\end{proof}
The following is a consequence of Theorem~\ref{Euler3graph}.
\begin{corollary}\label{thisstr}
Let $S$ be a Seidel matrix of order $n$ with exactly three distinct eigenvalues $\lambda$, $\mu$, and $\nu$. Then
\begin{enumerate}[$($a$)$]
\item $(n-1)(\lambda+\mu+\nu)+\lambda\mu\nu\equiv n^2+n+2\ (\mathrm{mod}\ 4)$\emph{;}
\item $(n-2)(\lambda+\mu+\nu)+\lambda\mu+\mu\nu+\nu\lambda\equiv n^2+n+1\ (\mathrm{mod}\ 4)$\emph{;}
\item $\lambda\mu+\mu\nu+\nu\lambda\equiv 1\ (\mathrm{mod}\ 2)$.
\end{enumerate}
\end{corollary}
\begin{proof}
We can assume, by Theorem \ref{Euler3graph}, that $S=J-2A-I$, where $A$ is the adjacency matrix of some Euler graph $\Gamma$. Therefore $AJ+JA\equiv 0\ (\mathrm{mod}\ 2)$. Consider equation \eqref{CHTv2} modulo $4$ and note that $JAJ=2eJ$, where $e$ is the number of edges in $\Gamma$. Whence, we have $S^3\equiv2A-I+(n^2+n-1)J\ (\mathrm{mod}\ 4)$. Comparing this with equation~\eqref{CHT} and replacing $S^2$ by $(J-2A-I)^2$ yields
\begin{equation*}
\begin{split}
&(\lambda+\mu+\nu)((n-2)J+I)-(\lambda\mu+\mu\nu+\nu\lambda)(J-2A-I)+\lambda\mu\nu I\\
&\equiv2A-I+(n^2+n-1)J\ (\mathrm{mod}\ 4).
\end{split}
\end{equation*}
Consider the $(i,j)$th entry of this congruence. By taking $i=j$ we obtain part~(a). By taking two distinct off-diagonal coordinates $i\neq j$ and $k\neq \ell$, for which $a_{ij}=0$ and $a_{k\ell}=1$, respectively, we obtain a system of two congruences, being equivalent to part (b) and (c).
\end{proof}
Next we prove that if $S$ is a Seidel matrix with exactly three distinct eigenvalues, then the multiplicities of these cannot be all equal.
\begin{proposition}\label{S3prop3}
Let $S$ be a Seidel matrix of order $n$ with spectrum $\{\left[\lambda\right]^{n-d}, \left[\mu\right]^m$, $\left[\nu\right]^{d-m}\}$ where $2\leq d\leq n-1$ and $1\leq m\leq d-1$. Then the case $m=n-d=d-m$ is impossible.
\end{proposition}
\begin{proof}
Suppose the contrary, i.e., there exists a Seidel matrix where $m=n-d=d-m=n/3$. It follows that $n\equiv 0\ (\mathrm{mod}\ 3)$. Using equations \eqref{St1} and \eqref{St2} we find that $\lambda+\mu+\nu=0$ and $\lambda\mu+\mu\nu+\nu\lambda=-3(n-1)/2$. In particular, $n$ is odd. Plugging these into part~(a) of Corollary~\ref{thisstr} we find that $n\equiv 3\ (\mathrm{mod}\ 4)$. Now we can put $n=12M+3$ and use Corollary~\ref{S2mC} to infer that $\mathrm{det}S=(\lambda\mu\nu)^{4M+1}\equiv 1-n\equiv2\ (\mathrm{mod}\ 4)$. This forces $\lambda\mu\nu$ to be even and $M=0$, and since there is no such $3\times 3$ Seidel matrix, we have a contradiction.
\end{proof}
\begin{corollary}\label{cubicfff}
Every Seidel matrix with exactly three distinct eigenvalues has an integer eigenvalue.
\end{corollary}
\begin{proof}
Follows immediately from Proposition~\ref{S3prop3}.
\end{proof}
The following is a generalization of Proposition~\ref{s2p23}.
\begin{theorem}\label{3g}
Let $d\geq 1$ and let $S$ be a Seidel matrix of order $n\geq 2$ with smallest eigenvalue $\lambda_0$ of multiplicity $n-d\geq1$. Assume that $S$ has another eigenvalue $\mu$ of multiplicity $m\leq d$. Then
\beql{1}
\left|\mu+\frac{\lambda_0(n-d)}{d}\right|\leq\frac{\sqrt{n(d(n-1)-\lambda_0^2(n-d))}}{d}\cdot\sqrt{\frac{d-m}{m}}.
\eeq
Equality holds if and only if $S$ has at most three distinct eigenvalues.
\end{theorem}
The proof is standard, and follows the same spectral analysis that appears in \cite{vLS}.
\begin{proof}
First note that the right hand side is well-defined by \cite[Lemma 6.1]{vLS}. Let us denote the remaining eigenvalues of $S$ by $\lambda_1,\lambda_2,\hdots,\lambda_{d-m}$. Then, by using that $\mathrm{tr}S=0$ and $\mathrm{tr}S^2=n(n-1)$, we find, after an application of the Cauchy--Schwartz inequality, that
\begin{align*}\left((n-d)\lambda_0+m\mu\right)^2&=\left(\sum_{i=1}^{d-m}\lambda_i\right)^2\\
&\leq(d-m)\sum_{i=1}^{d-m}\lambda_i^2=(d-m)\left(n(n-1)-(n-d)\lambda_0^2-m\mu^2\right).
\end{align*}
The result follows after some algebraic manipulations.
\end{proof}
If $S$ is a Seidel matrix with exactly three distinct eigenvalues, then Theorem~\ref{3g} boils down to an algebraic identity, which in turn can be used to tabulate feasible parameter sets of such Seidel matrices. Some interesting parameter sets are shown in Appendix~\ref{Ast}.

Now we turn to a description of the case when two of the three multiplicities are equal.
\begin{proposition}\label{S3T47}
Let $S$ be a Seidel matrix of order $n$, let $(n+1)/2\leq d\leq n-1$, and suppose that $S$ has smallest eigenvalue $\lambda_0$ of multiplicity $n-d$. Assume further that $S$ has two additional eigenvalues $\mu\in\mathbb{Z}$ and $\nu$ of multiplicity $2d-n$ and $n-d$, respectively. Then 
\[\left\{\lambda_0,\nu\right\}=\left\{\frac{(n-2d)\mu \mp\sqrt{n (2 (n - 1) (n - d) +  (n - 2 d)\mu^2)}}{2 (n-d)}\right\},\]
with $\lambda_0<\nu$, and hence necessarily $d\mu\equiv 0\ (\mathrm{mod}\ n-d)$, and
\[n(n-1)/2+(1+(-1)^\mu)d\mu/4\equiv 0\ (\mathrm{mod}\ n-d).\]
\end{proposition}
Note that Proposition~\ref{S3T47} covers all cases where two out of the three eigenvalue multiplicities are equal, up to taking the negative of the Seidel matrix if it is necessary.
\begin{proof}
The values of $\lambda_0$ and $\nu$ follow from Theorem~\ref{3g}, while the necessary conditions can be obtained as follows. On the one hand $\lambda_0+\nu$ is an algebraic integer, hence we readily see that there exists an integer $M$ such that $d\mu=M(n-d)$. On the other hand $\lambda_0\nu$ is also an algebraic integer, and therefore $-2(n-d)\lambda_0 \nu = n(n-1) + d \mu (\mu - M)$. Hence, we have
\[n(n-1)+d\mu(\mu-1)\equiv M(M-1)(n-d)\equiv0\ (\mathrm{mod}\ 2(n-d)).\]
For $\mu$ odd $\mu-1$ is even and the statement follows after recalling $d\mu=M(n-d)$. Otherwise, if $\mu$ is even then we can rewrite the left hand side as $n(n-1)+d\mu(\mu-2)+d\mu=n(n-1)+M(n-d)(\mu-2)+d\mu$, and the statement follows.
\end{proof}
An infinite family of examples where the multiplicity of two eigenvalues agree can be obtained from the symmetric Paley matrices of prime power order $n\equiv 1\ (\mathrm{mod}\ 4)$; the spectrum of these matrices is $\{\left[-\sqrt{n}\right]^{(n-1)/2},\left[0\right]^1,\left[\sqrt{n}\right]^{(n-1)/2}\}$. There exists an additional example of order $n=12\equiv 0\ (\mathrm{mod}\ 4)$ with spectrum $\{\left[-1-2\sqrt5\right]^3,[1]^6,\left[-1+2\sqrt5\right]^3\}$, coming from the icosahedron. Note that in all of the previous examples the integer eigenvalue is neither the smallest, nor the largest one. This is not the case in general, as there exists a Seidel matrix of order $n=10\equiv 2\ (\mathrm{mod}\ 4)$ with spectrum $\{[-3]^{4},[2-\sqrt5]^3,[2+\sqrt5]^3\}$. Further examples with quadratic irrational eigenvalues come from Proposition~\ref{existv1}, but none of them is of order $n\equiv 3\ (\mathrm{mod}\ 4)$, as we show next.
\begin{lemma}\label{quadr}
Let $S$ be a Seidel matrix of order $n$ with exactly three distinct eigenvalues such that two of them have the same multiplicity. Then $n\not\equiv 3\ (\mathrm{mod}\ 4)$.
\end{lemma}
\begin{proof}
Assume to the contrary that $S$ is a Seidel matrix of order $n\equiv 3\ (\mathrm{mod}\ 4)$ with exactly three eigenvalues out of which two have the same multiplicity. By considering $-S$ instead, if necessary, $S$ has spectrum $\{[\lambda]^{n-d},[\mu]^{2d-n},[\nu]^{n-d}\}$, where the values are given by Proposition~\ref{S3T47}. By Corollary~\ref{S2mC} we see that $\left(\lambda\nu\right)^{n-d}\mu^{2d-n}\equiv 2\ (\mathrm{mod}\ 4)$, and therefore either $\mu\equiv 2\ (\mathrm{mod}\ 4)$ and $d=(n+1)/2$ or $\mu$ is odd and $d=n-1$. In the first case we find that $\lambda\nu=(n+1)\mu^2/(n-1)^2-n$ and hence $(n-1)$ divides $\mu$ for elementary divisibility reasons. Therefore, by the Gershgorin circle theorem, $\mu^2=(n-1)^2$. In particular $\lambda=\nu$, a contradiction. Otherwise, from the discriminant appearing in Proposition~\ref{S3T47}, we must have $2 (n - 1) (n - d) +  (n - 2 d)\mu^2\geq 0$. Hence $\mu^2=1$, and we find that $(\lambda-\mu)(\nu-\mu)=0$, a contradiction.
\end{proof}
We decide all admissible orders of Seidel matrices with exactly three distinct eigenvalues.
\begin{theorem}\label{vacuous}
There exists a Seidel matrix of order $n\geq 3$ with exactly three distinct eigenvalues if and only if $n\neq 4$ and $n$ is not a prime $p\equiv 3\ (\mathrm{mod}\ 4)$.
\end{theorem}
\begin{proof}
The composite case follows from the Seidel matrices $J_a\otimes (J_b-2I_b)+I_{ab}$, apart from order $n=4$, in which case no such Seidel matrix exists. For primes $p\equiv 1\ (\mathrm{mod}\ 4)$ the Paley matrices described above form a family of examples. Therefore our goal is to show that Seidel matrices of prime order $p\equiv 3\ (\mathrm{mod}\ 4)$ having exactly three eigenvalues do not exist. Assume that this is not the case: fix a prime $p\equiv 3\ (\mathrm{mod}\ 4)$ and consider such a Seidel matrix $S$. Note that $S$ has integer eigenvalues by Corollary~\ref{cubicfff} and Lemma~\ref{quadr}. It follows from Corollary~\ref{S2mC} that $S$ has spectrum $\{\left[\lambda\right]^{p-d}, \left[\mu\right]^1,\left[\nu\right]^{d-1}\}$ for some $2\leq d\leq p-1$ with $\mu/2, \lambda$, and $\nu$ all being odd. By eliminating $\mu$ from the standard equations \eqref{St1} and \eqref{St2} (where $n=p$ and $m=1$) we infer that $d(d-1)(\lambda-\nu)^2\equiv 0\ (\mathrm{mod}\ p)$, and hence $\lambda\equiv \nu\ (\mathrm{mod}\ p)$. Since they are both odd and at most $p-1$ by the Gershgorin circle theorem, $\lambda=\nu$ follows, a contradiction.
\end{proof}
\subsection{Constructions and some sporadic examples}\label{subs2}
In this section we give various constructions of Seidel matrices with exactly three distinct eigenvalues. As shown in the next result, examples of such Seidel matrices come from connected regular graphs with three or four distinct eigenvalues. Note that connected regular graphs with three eigenvalues are known as \emph{strongly regular graphs} \cite{HB, vDS}.
\begin{example}\label{rank3g}
Let $\Gamma$ be a strongly regular graph with spectrum $\{\left[2\right]^{24},\left[-4\right]^{15},\left[12\right]^1\}$ $($see \cite{TS2}$)$. Then from Lemma~\ref{triww2} we obtain a Seidel matrix with spectrum $\{\left[-5\right]^{24}$, $\left[7\right]^{15}, \left[15\right]^1\}$. The corresponding $40$ lines form the largest known equiangular line system in $\mathbb{R}^{16}$.
\end{example}
The next result is a generalization of \cite[Construction 6.2]{vLS}.
\begin{proposition}\label{existv1}
Let $S$ be a Seidel matrix with spectrum $\{[\lambda_0]^{a-c}, [\lambda_1]^c\}$. Then for all $b\geq 2$ there exists a Seidel matrix with spectrum $\{[1-(1-\lambda_0)b]^{a-c}, [1]^{a(b-1)},$ $[(\lambda_1-1)b+1]^c\}$.
\end{proposition}
\begin{proof}
The Seidel matrix $J_b\otimes (S-I_a)+I_{ab}$ of order $ab$ has the desired spectrum.
\end{proof}
The following, besides being a folklore technical result, serves as another example of Seidel matrices with three distinct eigenvalues.
\begin{lemma}\label{S4L8}
Let $S$ be a Seidel matrix of order $n\geq 2$ with spectrum $\{\left[\lambda_0\right]^{n-d},\left[\lambda_1\right]^d\}$ for some $d$ with $1\leq d\leq n-1$. Let $S'$ be a principal $(n-1)\times (n-1)$ submatrix of $S$. Then the spectrum of $S'$ is $\{\left[\lambda_0\right]^{n-d-1},\left[\lambda_1\right]^{d-1},\left[\lambda_0+\lambda_1\right]^1\}$.
\end{lemma}
\begin{proof}
By interlacing (see Lemma~\ref{intL}) $S'$ has spectrum $\{\left[\lambda_0\right]^{n-d-1},\left[\lambda_1\right]^{d-1},\left[x\right]^1\}$ for some $x$ with $\lambda_0\leq x\leq \lambda_1$. Since $\mathrm{tr}S=\mathrm{tr}S'=0$, it follows that $x=\lambda_0+\lambda_1$, as claimed.
\end{proof}
\begin{proposition}\label{S4P9}
Let $S$ be a Seidel matrix of order $n\geq 2$ with spectrum $\{\left[\lambda_0\right]^{n-d}$, $\left[\lambda_1\right]^d\}$ for some $d$ with $1\leq d\leq n-1$ and assume that for some $c$ with $1\leq c\leq d$ it admits the block partition $S=\left[\begin{smallmatrix}
J_c-I_c & \ast\\
\ast & S'
\end{smallmatrix}\right]$. Then $\mathrm{tr}\left((S')^3\right)=(n-3c)(n-1)(\lambda_0+\lambda_1)+c(c-1)(3\lambda_0+3\lambda_1-c+2)$.
\end{proposition}
\begin{proof}
For $c=1$ the spectrum is known by Lemma~\ref{S4L8}, while for $c=2$ it can be determined from the standard equations \eqref{St1} and \eqref{St2}, and hence the claim follows in these cases. Therefore we assume that $c\geq 3$ and compare $\mathrm{tr}((S')^3)$ with $\mathrm{tr}(S^3)$ to find:
\begin{equation}\label{origTr}
\begin{split}\mathrm{tr}(S^3)&=3\sum_{M\in\mathcal{M}}\mathrm{det}M\\
&=\mathrm{tr}\left((S')^3\right)+3\sum_{M\in\mathcal{M}_1}\mathrm{det}M+3\sum_{M\in\mathcal{M}_2}\mathrm{det}M+3\sum_{M\in\mathcal{M}_3}\mathrm{det}M,
\end{split}
\end{equation}
where $\mathcal{M}$ is the subset of all $3\times 3$ principal submatrices of $S$, while $\mathcal{M}_i$, $i=1,2,3$ are those subsets of $\mathcal{M}$ whose members intersect the first $c$ rows of $S$ in exactly $i$ places. We calculate these sums separately.

The last term is straightforward, as we have $\binom{c}{3}$ minors within the first $c\times c$ submatrix of $S$, each contributing the value $2$ to the sum:
\beql{x1x}\sum_{M\in\mathcal{M}_3}\mathrm{det}M=2\binom{c}{3}.\eeq

The value of the sum over $\mathcal{M}_2$ can be obtained by evaluating the inner products of the rows of $S$, $\left\langle S_i,S_j\right\rangle$, via two different ways. First, since $S$ has two eigenvalues only, it satisfies 
\beql{ggg}S^2=(\lambda_0+\lambda_1)S+(n-1)I\eeq
and consequently for every $i$ and $j$ with $1\leq i<j\leq c$ we have $\left\langle S_i,S_j\right\rangle=\left[S^2\right]_{ij}=\lambda_0+\lambda_1$. Second, we count the number of vertical pairs $[1,1]^T$, $[1,-1]^T$, $[-1,1]^T$ and $[-1,-1]^T$ within the $i$th and $j$th rows of $S$, which we denote by $w_{i,j}$, $x_{i,j}$, $y_{i,j}$, and $z_{i,j}$, respectively. It follows that $\left\langle S_i,S_j\right\rangle=w_{i,j}-x_{i,j}-y_{i,j}+z_{i,j}$. By combining these two equations for the inner products we find that the quantity $w_{i,j}-x_{i,j}-y_{i,j}+z_{i,j}$ does not depend on the subscripts, and hence
\beql{fff}\sum_{M\in\mathcal{M}_2}\mathrm{det}M=2\sum_{\substack{i,j=1\\ i<j}}^c\left(w_{i,j}-x_{i,j}-y_{i,j}+z_{i,j}-(c-2)\right)=c(c-1)(\lambda_0+\lambda_1-c+2).\eeq

The value of the final sum can be obtained by summing up the signed closed $3$-walks originating from the first $c$ vertices of $S$ in two ways. Firstly, by counting, double counting, and triple counting, we have
\begin{align*}\sum_{i=1}^c\left[S^3\right]_{ii}&=\sum_{i=1}^c\sum_{k,l}s_{ik}s_{kl}s_{li}=\sum_{k,l}s_{kl}\sum_{i=1}^cs_{ik}s_{li}\\
&=\sum_{M\in\mathcal{M}_1}\mathrm{det}M+2\sum_{M\in\mathcal{M}_2}\mathrm{det}M+3\sum_{M\in\mathcal{M}_3}\mathrm{det}M.
\end{align*}
Secondly, we know from \eqref{ggg} that $\left[S^3\right]_{ii}=(n-1)(\lambda_0+\lambda_1)$ for all $1\leq i<j\leq n$ and hence
\[\sum_{i=1}^c\left[S^3\right]_{ii}=c(n-1)(\lambda_0+\lambda_1).\]
Comparing these two we find that
\beql{fin}\sum_{M\in\mathcal{M}_1}\mathrm{det}M=c(n-1)(\lambda_0+\lambda_1)-2\sum_{M\in\mathcal{M}_2}\mathrm{det}M-3\sum_{M\in\mathcal{M}_3}\mathrm{det}M,\eeq
and therefore combining \eqref{x1x}, \eqref{fff}, and \eqref{fin} with \eqref{origTr} yields
\begin{align*}\mathrm{tr}\left((S')^3\right)&=\mathrm{tr}(S^3)-3c(n-1)(\lambda_0+\lambda_1)+3\sum_{M\in\mathcal{M}_2}\mathrm{det}M+6\sum_{M\in\mathcal{M}_3}\mathrm{det}M\\
&=(n-3c)(n-1)(\lambda_0+\lambda_1)+c(c-1)(3\lambda_0+3\lambda_1-c+2).\qedhere
\end{align*}
\end{proof}
Now we can generalize Lemma~\ref{S4L8} as follows.
\begin{proposition}\label{newPP}
Let $S$ be a Seidel matrix of order $n\geq 2$ with spectrum $\{\left[\lambda_0\right]^{n-d}$, $\left[\lambda_1\right]^d\}$ for some $d$ with $1\leq d\leq n-1$ and assume that for some $c$ with $1\leq c\leq\min\{d,n-d\}$ it admits the block partition $S=\left[\begin{smallmatrix}
J_c-I_c & \ast\\
\ast & S'
\end{smallmatrix}\right]$. Then the spectrum of $S'$ is $\{\left[\lambda_0\right]^{n-d-c},\left[\lambda_1\right]^{d-c},\left[\lambda_0+\lambda_1+1-c\right]^1,\left[\lambda_0+\lambda_1+1\right]^{c-1}\}$.
\end{proposition}
\begin{rem}\label{r33}
If $J_c-I_c$ is a principal submatrix of a Seidel matrix $S$ of order $n\geq2$ with largest eigenvalue $\lambda_1$, then $c\leq\min\{d,n-d\}$ unless $S$ is switching equivalent to $J-I$. Another bound, $c\leq \lambda_1+1$ follows from interlacing.  
\end{rem}
\begin{proof}
We prove by induction on $c$. For $c=1$ the claim follows from Lemma~\ref{S4L8} and, for $c=2$, it is easy to determine the spectrum from the standard equations \eqref{St1} and \eqref{St2}. Therefore we assume that $c\geq 3$. By removing the first $c-1$ rows and columns from $S$ we obtain, by our induction hypothesis, a Seidel matrix whose spectrum is the following: 
$\{\left[\lambda_0\right]^{n-d-c+1},\left[\lambda_1\right]^{d-c+1},\left[\lambda_0+\lambda_1+2-c\right]^1,\left[\lambda_0+\lambda_1+1\right]^{c-2}\}$, where all the multiplicities are at least $1$. We now remove the $c$th row and column from $S$ as well and use interlacing to find that the spectrum of $S'$ reads $\{\left[\lambda_0\right]^{n-d-c},\left[\lambda_1\right]^{d-c},\left[\lambda_0+\lambda_1+1\right]^{c-3},\left[x\right]^1,\left[y\right]^1,\left[z\right]^1\}$ for some real numbers $x$, $y$, and $z$. By the standard equations and Proposition~\ref{S4P9} we find that
\begin{align*}
\mathrm{tr}(S')&=(n-d-c)\lambda_0+(d-c)\lambda_1+(c-3)(\lambda_0+\lambda_1+1)+x+y+z=0,\\
\mathrm{tr}\left((S')^2\right)&=(n-d-c)\lambda_0^2+(d-c)\lambda_1^2+(c-3)(\lambda_0+\lambda_1+1)^2+x^2+y^2+z^2\\
&=(n-c)(n-c-1),\\
\mathrm{tr}\left((S')^3\right)&=(n-d-c)\lambda_0^3+(d-c)\lambda_1^3+(c-3)(\lambda_0+\lambda_1+1)^3+x^3+y^3+z^3\\
&=(n-3c)(n-1)(\lambda_0+\lambda_1)+c(c-1)(3\lambda_0+3\lambda_1-c+2).
\end{align*}
Now by noting that $\lambda_1=-(n-d)\lambda_0/d$ and $\lambda_0^2=d(n-1)/(n-d)$ it is immediate to check that $\{x,y,z\}=\{\lambda_0+\lambda_1+1-c,\lambda_0+\lambda_1+1,\lambda_0+\lambda_1+1\}$ (in any order) is a solution. Since the elementary symmetric polynomials $x+y+z$, $xy+yz+zx$ and $xyz$ can be expressed using the equations above, it follows that this is the only solution.
\end{proof}
The following is the main result of this section.
\begin{theorem}[cf.~\mbox{\cite[Lemma~2.1]{HxD}}]\label{S4MC}
Let $S$ be a Seidel matrix of order $n\geq 2$ with spectrum $\{\left[\lambda_0\right]^{n-d}$, $\left[\lambda_1\right]^{d}\}$ with $\lambda_1\leq\min\{d-1,n-d-1\}$ and assume that it admits the block partition $S=\left[\begin{smallmatrix}
J_{\lambda_1+1}-I_{\lambda_1+1} & \ast\\
\ast & S'
\end{smallmatrix}\right]$. Then the spectrum of $S'$ is $\{\left[\lambda_0\right]^{n-d-\lambda_1}$, $\left[\lambda_1\right]^{d-\lambda_1-1}$, $\left[\lambda_0+\lambda_1+1\right]^{\lambda_1}\}$.
\end{theorem}
\begin{proof}
Setting $c=\lambda_1+1$ in Proposition~\ref{newPP} and observing that in this case $\lambda_0+\lambda_1+1-c=\lambda_0$ we obtain the desired result.
\end{proof}
Note that Theorem~\ref{S4MC} describes a method to obtain equiangular lines in a smaller dimensional space by removing a maximal independent set (see Remark~\ref{r33}) from the underlying graph. One consequence is the following result of Tremain.
\begin{example}[cf.\ \cite{JCT}]\label{ex415}
We construct $28$ equiangular lines in $\mathbb{R}^{14}$ with common angle $1/5$ from a regular two-graph on $36$ vertices. Consider the $35\times 35$ matrix $S$ in \cite[p.~ 75]{H36}. Let $S'$ be constructed from $-S$ by extending it with a $36$th row and column with all entries equal to $1$ except for $S'_{36,36}=0$. Upon removing rows $\{1, 6, 11, 17, 20, 33, 34, 36\}$ along with the same set of columns from $S'$ we find, by Theorem~\ref{S4MC}, that the spectrum of the resulting Seidel matrix of order $28$ is $\{\left[-5\right]^{14},\left[3\right]^7,\left[7\right]^7\}$. Additional inequivalent examples might be obtained from other regular two-graphs on $36$ vertices; these have been classified in \cite{MSP}.
\end{example}
Next we discuss some relevant, sporadic examples of equiangular line systems.
\begin{example}[see \cite{LS}]\label{STS19}
Let $\mathcal{B}:=\{\{4^i+j,7\cdot4^i+j,11\cdot 4^i+j\}\ (\mathrm{mod}\ 19)\colon 0\leq i\leq 2, 0\leq j\leq 18\}$ be the set of blocks of the Netto triple system on the ground set $X=\{1,2,\hdots,19\}$. Let $e_i$ denote the standard basis vectors in $\mathbb{R}^{19}$ for $1\leq i\leq 19$ and for a subset $T\subseteq X$ let us denote $e_T:=\sum_{i\in T}e_i$. The vectors $v_B:=(6e_B+e_1-e_X)/\sqrt{90}$ for which $1\notin B\in\mathcal{B}$ are all orthogonal to both $e_1$ and $e_X$ and form an equiangular line system of $48$ lines in $\mathbb{R}^{17}$. Moreover, the corresponding Seidel matrix has spectrum $\{\left[-5\right]^{31}, \left[7\right]^{8}, \left[11\right]^9\}$.
\end{example}
The known maximal set of equiangular lines in dimensions $19$--$23$ all come from the Witt design. The examples in dimension $21$, $22$, and $23$ are regular two-graphs; Taylor's example \cite{TTH} in dimension $20$ has spectrum $\{[-5]^{70},[15]^9,[19]^{10},[25]^1\}$; while the following construction, discovered by Asche, leads to a Seidel matrix with three distinct eigenvalues in dimension $19$.
\begin{example}[see \mbox{\cite[p.~ 124]{TTH}}]\label{ASCh}
Let $\mathcal{B}$ be the set of $759$ blocks of the Witt design, the ``octads'', defined on the ground set $X=\{1,2,\hdots,24\}$. Let $e_i$ denote the standard basis vectors in $\mathbb{R}^{24}$ for $1\leq i\leq 24$ and for a subset $T\subseteq X$ let us denote $e_T:=\sum_{i\in T}e_i$. Let $B_1, B_2\in\mathcal{B}$ such that $1\notin B_1, B_2$ and $B_1\cap B_2=\{2,3\}$. The vectors $v_B:=(4e_B-4e_1-e_X)/\sqrt{80}$ for which $1\in B\in\mathcal{B}$ are all orthogonal to $4e_1+e_X$. Those, which in addition are orthogonal to all of $e_1-e_2$, $e_1-e_3$, $v_{B_1}$, and $v_{B_2}$, form an equiangular line system of $72$ lines in $\mathbb{R}^{19}$. Moreover, the corresponding Seidel matrix has spectrum $\{\left[-5\right]^{53}, \left[13\right]^{16}, \left[19\right]^3\}$.
\end{example}%
%
The following observation nicely connects the existence of some hypothetical strongly regular graphs on $76$ and $96$ vertices, corresponding to some of the upper bounds in Table \ref{xyz}.
\begin{rem}\label{S4C2}
One can see the following by repeated application of Theorem~\ref{S4MC}. Assume that there exists a Seidel matrix of order $96$ with spectrum $\{\left[-5\right]^{76},\left[19\right]^{20}\}$ containing a $J_{20}-I_{20}$ principal minor (see Remark~\ref{r33}). Then, there exists a Seidel matrix of order $76$ with spectrum $\{\left[-5\right]^{57},\left[15\right]^{19}\}$. If, in addition, this latter matrix contains a $J_{16}-I_{16}$ principal minor, then there further exists a Seidel matrix of order $60$ with spectrum $\{\left[-5\right]^{42},\left[11\right]^{15},\left[15\right]^{3}\}$. These examples would improve upon the best known lower bounds on the number of equiangular lines in dimensions $20$, $19$, and $18$. Regular two-graphs with maximal cliques were considered in \cite{GMS}.
\end{rem}
In Table \ref{uptotw} we display the number of inequivalent Seidel matrices with exactly three distinct eigenvalues up to order $n\leq 12$ (cf.\ Theorem~\ref{vacuous} and the results of Section~\ref{sec5}).
\[\begin{array}{c|cccccccccc}
n & 3 & 4 & 5 & 6 & 7 & 8 & 9 & 10 & 11 & 12\\
\hline
  & 0 & 0 & 1 & 2 & 0 & 2 & 3 &  4 & 0  & 10
\end{array}\]
\begingroup
\captionof{table}{The number of Seidel matrices of order $n$ with exactly three distinct eigenvalues, up to switching equivalence.}\label{uptotw}
\endgroup
\subsection{A nonexistence result}\label{subs3}
Let $\lambda_0<\lambda_1\leq \lambda_2\leq \dots\leq\lambda_{d}$ be the eigenvalues of a Seidel matrix of order $n$, and let $\mu\neq\lambda_0$ be an integer. In the following theorem we estimate the quantity $\sum_{i=1}^d(\lambda_i-\mu)^2$ in a nonstandard way (cf.\ Theorem~\ref{3g}, where we have obtained an upper bound on $|\mu+\lambda_0(n-d)/d|$), and prove that in case of equality the corresponding Seidel matrix must have at most four distinct eigenvalues. This provides us with a crucial insight into the structure of certain Seidel matrices which we can further analyze in detail.
\begin{theorem}[cf.\ \mbox{Theorem~\ref{3g}}]\label{maintnon}
Let $d\geq 1$, let $S$ be a Seidel matrix of order $n\geq 2$ with smallest eigenvalue $\lambda_0$ of multiplicity $n-d\geq 1$ and let $\mu\neq \lambda_0$ be an integer satisfying $\mathrm{dim}(\mathrm{ker}\left(S-\mu I\right))=m$. Assume further that $\lambda_0\leq-\sqrt{d(n^2-n+m-d)/(n^2-dn)}$. Then
\begin{equation}\label{superduper}
\left|\mu+\frac{\lambda_0(n-d)}{d}\right|\geq\frac{\sqrt{d^2-d \left(m+n \left(\lambda_0^2+n-1\right)\right)+\lambda_0^2
   n^2}}{d},
\end{equation} with equality if and only if $S$ has spectrum $\{[\lambda_0]^{n-d},[\mu-1]^w,[\mu]^m,[\mu+1]^{d-m-w}\}$, where $w=((n-d)\lambda_0+d\mu+d-m)/2$.
\end{theorem}
\begin{proof}
Let $\lambda_1\leq \lambda_2\leq \dots\leq\lambda_{d}$ be the other eigenvalues of $S$. We have the following chain of (in)equalities:
\[n(n-1)+\mu^2n-(\lambda_0-\mu)^2(n-d)=\sum_{i=1}^d(\lambda_i-\mu)^2\geq (d-m)\prod_{\substack{i=1\\ \lambda_i\neq \mu}}^d(\lambda_i-\mu)^{2/d}\geq d-m,\]
where the leftmost relation follows from the standard equations~\eqref{St1} and \eqref{St2}; the middle relation is the inequality of arithmetic and geometric means; and the rightmost relation follows from the fact that the set $\{\lambda_i-\mu\colon 1\leq i\leq d, \lambda_i\neq \mu\}$ consists of nonzero algebraic integers closed under algebraic conjugation. The claim follows after simple algebraic manipulations. In the case of equality, $\lambda_i=\mu\pm 1$ for all $1\leq i\leq d$, for which $\lambda_i\neq \mu$. The values of $\mu$ and $w$ can be obtained from the standard equations.

Conversely, assume that $S$ is a Seidel matrix of order $n$ having spectrum $\{[\lambda_0]^{n-d}$, $[\mu-1]^w$, $[\mu]^m,[\mu+1]^{d-m-w}\}$, where $w$ is as defined in the theorem. Then the claimed inequality holds with equality.
\end{proof}
\begin{rem}\label{stupidremark}
We remark that if there is equality in \eqref{superduper} in Theorem~\ref{maintnon} and, in addition, $3\leq d\leq\lambda_0^2-2$, then $n=\left\lfloor d(\lambda_0^2-1)/(\lambda_0^2-d)\right\rfloor$. Indeed, \eqref{superduper} implies that $\mu$ must satisfy a quadratic equation with discriminant $d^2-dn\lambda_0^2-dn^2+dn+n^2\lambda_0^2-dm\geq 0$. In particular $d^2-dn\lambda_0^2-dn^2+dn+n^2\lambda_0^2\geq 0$. Combining this latter inequality with the relative bound given in Proposition~\ref{s2p23} gives us
\[L:=\frac{d(\lambda_0^2-1)+d\sqrt{(\lambda_0^2-3)^2+4d-8}}{2(\lambda_0^2-d)}\leq n\leq \frac{d(\lambda_0^2-1)}{\lambda_0^2-d}=:R.\]
Since $R-L=2d/(\lambda_0^2-1+\sqrt{(\lambda_0^2-3)^2+4d-8})<1$, the claim follows. In particular, $n$ is as close to the relative bound as it can be.
\end{rem}
\begin{example}\label{DRJACK}
Here we illustrate the utility of Theorem~\ref{maintnon} by showing that there do not exist $n=22$ lines in $\mathbb{R}^{12}$ with common angle $1/5$. Assume that $S$ is the Seidel matrix of such a configuration with smallest eigenvalue $\lambda_0=-5$. Let $\mu:=4$ and observe that Corollary~\ref{S2mC} implies that $\mu$ cannot be an eigenvalue of $S$. Therefore $m=0$, which contradicts \eqref{superduper}.
\end{example}
It follows from Theorem~\ref{maintnon} (see Corollary~\ref{seethisthing}) that $30$ lines in $\mathbb{R}^{14}$ must correspond to a Seidel matrix with spectrum $\{[-5]^{16},[5]^9,[7]^5\}$. Similarly, $42$ lines in $\mathbb{R}^{16}$ must come from a Seidel matrix with spectrum $\{[-5]^{26},[7]^7,[9]^9\}$. We prove that these Seidel matrices do not exist.
\begin{lemma}\label{strlem1}
Let $M$ be a positive semi-definite $\{0,\pm1\}$ matrix of order $n$ with constant diagonal entries $1$. Then there exist positive integers $c$, $k_1,\hdots,k_c$, such that $M$ is switching equivalent to the block diagonal matrix $\mathrm{diag}\left[J_{k_1},\hdots,J_{k_c}\right]$.
\end{lemma}
\begin{proof}
The proof goes by induction on $n$. For $n=1$ there is nothing to prove, so we can consider an $(n+1)\times (n+1)$ matrix $M'$. We can assume, by the inductive hypothesis, that its leading principal $n\times n$ submatrix $M$ is already in the desired form, namely $M=\mathrm{diag}\left[J_{k_1},\hdots,J_{k_c}\right]$ for some positive integers $c$, $k_1,\hdots,k_c$. Since $M'$ is positive semi-definite, it follows that it cannot have any principal submatrices, switching equivalent to either
\[\begin{bmatrix}
1 & 0 & 1\\
0 & 1 & 1\\
1 & 1 & 1\\
\end{bmatrix}\qquad\text{or}\qquad\begin{bmatrix}
1 & 1 & 1\\
1 & 1 & -1\\
1 & -1 & 1\\
\end{bmatrix}.\]
Therefore either $M'$ is switching equivalent to $\mathrm{diag}\left[J_{k_1},\hdots,J_{k_c},J_1\right]$ or, up to reordering the blocks, $M'$ is switching equivalent to $\mathrm{diag}\left[J_{k_1},\hdots,J_{k_c+1}\right]$.
\end{proof}
\begin{theorem}\label{nonex}
Let $S$ be a Seidel matrix with spectrum $\{[\lambda]^{a},[\mu]^b,[\nu]^{c}\}$ of order $n\equiv 0\ (\mathrm{mod}\ 2)$. Assume that $\lambda+\mu\equiv n-2\ (\mathrm{mod}\ 4)$, and $|n-1+\lambda\mu|=4$. Then $|(\nu-\lambda)(\nu-\mu)|/4=n/c\in\mathbb{Z}$, and $|\nu|\leq n/c-1$.
\end{theorem}
\begin{proof}
Let $\varrho:=(\nu-\lambda)(\nu-\mu)$ and let $\sigma:=\mathrm{sgn}\varrho$. Consider the positive semi-definite matrix $M:=\sigma(S-\lambda I)(S-\mu I)$ with spectrum $\{[0]^{a+b},[|\varrho|]^c\}$. Let $A$ be the adjacency matrix of the ambient graph corresponding to $S$. By Theorem~\ref{Euler3graph} we may assume that $AJ+JA\equiv 0\ (\mathrm{mod}\ 2)$ and hence $S^2=(n-2)J-2(AJ+JA)+4(A^2+A)+I\equiv (n-2)J+I\ (\mathrm{mod}\ 4)$. Therefore the off-diagonal entries of $M$ are all divisible by $4$, while the diagonal entries of $M$ are all equal to $4=\sigma(n-1+\lambda\mu)\geq 0$. Since every $2\times 2$ submatrix of $M$ must be positive semi-definite, it follows that every off-diagonal entry $m$ satisfies $|m|\leq |n-1+\lambda\mu|=4$. Therefore $m\in\{0,\pm4\}$, and hence $M$ is a $\{0,\pm 4\}$-matrix, and $|\varrho|$ is divisible by $4$. By Lemma \ref{strlem1} we may assume that $M=4\mathrm{diag}\left[J_{k_1},\hdots,J_{k_c}\right]$ for some positive integers $k_1,\hdots,k_c$. Since each block must have an eigenvalue $|\varrho|/4$, it follows that $k_i=|\varrho|/4$ for all $i$, and thus $c|\varrho|/4=n$. Therefore $M=4J_{n/c}\otimes I_c$.
On the one hand, we find that
\[[S^3]_{11}=[\left(\sigma M+(\lambda+\mu)S-\lambda\mu I\right)S]_{11}=[(4\sigma J_{n/c}\otimes I_c)S]_{11}+(\lambda+\mu)(n-1).\]
On the other hand, from equation \eqref{CHT} we have $[S^3]_{11}=(\lambda+\mu+\nu)(n-1)+\lambda\mu\nu$. Therefore, we have $|\nu(n-1+\lambda\mu)|=4|[(J_{n/c}\otimes I_c)S]_{11}|\leq 4(n/c-1)$, and the result follows.
\end{proof}
\begin{corollary}[cf.\ \mbox{\cite[Appendix~A]{vDS}}]\label{nonex2}
There do not exist regular graphs with spectrum $\{[11]^1,[2]^{16},[-3]^9,[-4]^4\}$ or $\{[12]^1,[2]^{16},[-3]^8,[-4]^5\}$.
\end{corollary}
\begin{proof}
Either of these graphs would lead, via Lemma~\ref{triww2}, to a Seidel matrix with spectrum $\{[-5]^{16},[5]^9,[7]^5\}$, which does not exist by Theorem~\ref{nonex}.
\end{proof}
In \cite{vDS} the authors used a computer search to claim the nonexistence of regular graphs with spectrum $\{[11]^1,[2]^{16},[-3]^9,[-4]^4\}$. However, the case $\{[12]^1, [2]^{16}$, $[-3]^8, [-4]^5\}$ was left open.

By combining Theorems~\ref{maintnon} and \ref{nonex}, we obtain the main result of this subsection.
\begin{corollary}\label{seethisthing}
The maximum number of equiangular lines in $\mathbb{R}^{14}$ is at most $29$. The maximum number of equiangular lines in $\mathbb{R}^{16}$ is at most $41$.
\end{corollary}
\begin{proof}
We apply Theorem~\ref{maintnon} as follows. Assume that there exists $n=30$ ($n=42$, resp.) equiangular lines in dimension $d=14$ ($d=16$, resp.). Set $\mu$ to be the closest integer to $-\lambda_0(n-d)/d$. In both cases $\mu$ is even. Since, by Corollary~\ref{S2mC}, $S$ cannot have even eigenvalues in these cases, we get $m=0$. Therefore, in both cases, we have equality in \eqref{superduper}, and hence $S$ must be a Seidel matrix with three distinct eigenvalues. However, application of Theorem~\ref{nonex} shows that these Seidel matrices do not exist.
\end{proof}
A further implication of Theorem~\ref{maintnon} (combined with the first part of Theorem~\ref{s2t2}) is that $61$ equiangular lines in $\mathbb{R}^{18}$ must necessarily correspond to a Seidel matrix with spectrum $\{[-5]^{43}, [11]^{9},[12]^1,[13]^{8} \}$. We were unable to settle the existence of such matrices.
\section{Equiangular lines with common angle $1/5$}
Maximal equiangular line systems with common angle $\alpha=1/3$ are completely understood. In particular, the corresponding Seidel matrices must contain an $I_4-J_4$ principal submatrix \cite{LS}, and consequently one cannot have more than $2(d-1)$ such equiangular lines in $\mathbb{R}^d$ for $d\geq 15$ (see also \cite{HAA, vLS}). In this section we discuss results regarding $N_5(d)$, the maximum number of equiangular lines with common angle $1/5$. 
\begin{theorem}[see \cite{LS}]\label{thisLS}
Let $S$ be a Seidel matrix of order $n$ with smallest eigenvalue $-5$ of multiplicity $n-d$, containing an $I_6-J_6$ principal submatrix. Then $n\leq 276$ for $43\leq d\leq 185$, and $n\leq\left\lfloor 3(d-1)/2\right\rfloor$ for $d\geq 186$.
\end{theorem}
\begin{rem}
One out of the four Seidel matrices of order $12$ corresponding to maximal equiangular line systems in $\mathbb{R}^{9}$ with common angle $1/5$ does not contain (up to switching) any $I_6-J_6$ principal submatrix.
\end{rem}
Theorem~\ref{thisLS} was subsequently improved by Neumaier, who determined the maximum number of equiangular lines with common angle $1/5$ in $\mathbb{R}^d$ for large $d$ \cite{N}. It turns out that the Seidel matrix of all such line systems is switching equivalent to one whose ambient graph has largest eigenvalue at most $2$. Such graphs are called \emph{Dynkin graphs} \cite{LS}.
\begin{theorem}[Neumaier \cite{N}]\label{NTH}
There exists a positive integer $V$, for which if $S$ is a Seidel matrix of order $n\geq V$ with smallest eigenvalue $-5$, then $S$ is switching equivalent to some Seidel matrix $S'$ such that its ambient graph $\Gamma'$ is a Dynkin graph.
\end{theorem}
\begin{rem}
Neumaier \cite{N} claims, without proof, that $2486\leq V\leq 45374$.
\end{rem}
The case $t=2$ of the following technical result is the one we are interested in.
\begin{lemma}\label{previous}
Let $\Gamma$ be a graph with adjacency matrix $A$ having largest eigenvalue $t$ of multiplicity $m$ and let $S:=J-2A-I$ be the corresponding Seidel matrix. Then $S$ has smallest eigenvalue $\lambda_0\geq -(2t+1)$ with equality if and only if $m\geq 2$; in this case $\lambda_0$ has multiplicity $m-1$.
\end{lemma}
\begin{proof}
We have $x^TSx=x^TJx-2x^TAx-x^TIx\geq -(2t+1)$ for any unit vector $x$ and hence $\lambda_0\geq -(2t+1)$, as claimed. Moreover, equality holds if and only if $\mathrm{dim}(\mathrm{ker}J\cap\mathrm{ker}(A-tI))\geq1$. By Perron--Frobenius theory, $A$ has a nonnegative $t$-eigenvector, which cannot be a $0$-eigenvector of $J$, and consequently $\mathrm{dim}(\mathrm{ker}J+\mathrm{ker}(A-tI))=n$. Therefore
\begin{align*}
\mathrm{dim}(\mathrm{ker}(&S-(2t+1)I))=\mathrm{dim}(\mathrm{ker}J\cap\mathrm{ker}(A-tI))\\
&=\mathrm{dim}(\mathrm{ker}J)+\mathrm{dim}(\mathrm{ker}(A-tI))-\mathrm{dim}(\mathrm{ker}J+\mathrm{ker}(A-tI))\\
&=n-1+m-n=m-1.\qedhere
\end{align*}
\end{proof}
The following result was announced, without a proof, in \cite{N}.
\begin{corollary}\label{neumcor}
Let $d\geq \left\lceil(2V+5)/3\right\rceil$, where $V$ is given in Theorem~\ref{NTH}. Then $N_{5}(d)=\left\lfloor3(d-1)/2\right\rfloor$.
\end{corollary}
\begin{proof}
First we argue that $N_5(d)\geq\left\lfloor3(d-1)/2\right\rfloor$ for $d\geq 5$. For $d$ odd this immediately follows by setting $a=(d-1)/2$, $S=J_a-I_a$, and $b=3$ in Proposition~\ref{existv1}. For $d$ even consider the graph $\Gamma$ on $n=3(d-2)/2+1=\left\lfloor3(d-1)/2\right\rfloor$ vertices formed by $(d-2)/2$ disjoint triangles and an additional isolated node. By Lemma~\ref{previous} the smallest eigenvalue of the corresponding Seidel matrix $S$ is $-5$ of multiplicity $(d-2)/2-1$ and hence we have the desired configuration. Note that this construction is only interesting for $d\geq 186$.

Secondly we show that $N_5(d)\leq\left\lfloor3(d-1)/2\right\rfloor$ for $d\geq \left\lceil(2V+5)/3\right\rceil$. Suppose that $S$ is a Seidel matrix of the largest set of equiangular lines in $\mathbb{R}^d$ with common angle $1/5$. We estimate $n-d$, which is the multiplicity of the $-5$ eigenvalue of $S$ through the ambient graph $\Gamma$ as follows. From the first part and from the assumption on $d$ we find that $n\geq \left\lfloor3(d-1)/2\right\rfloor\geq V$. Therefore Theorem~\ref{NTH} applies and we may assume that the ambient graph $\Gamma$ with adjacency matrix $A=(J-S-I)/2$ is a Dynkin graph. By applying Lemma~\ref{previous}, it also follows that $A$ has largest eigenvalue at least $2$ of multiplicity $m=n-d+1$. Consequently $A$ has largest eigenvalue exactly $2$. Furthermore, since $m\leq\left\lfloor n/3\right\rfloor$, we have $n-d\leq\left\lfloor n/3\right\rfloor-1$ and we obtain $n\leq \left\lfloor 3(d-1)/2\right\rfloor$ as claimed.
\end{proof}
We conclude this section with an updated table containing bounds on $N_5(d)$.
\begin{theorem}
Bounds for the maximum number of equiangular lines in $\mathbb{R}^d$ with common angle $\alpha=1/5$ is given in Table \ref{newT} below.
\end{theorem}
{\arraycolsep=1.8pt
\[\begin{array}{c|cccccccccccccccccc}
d  & 2$--$4 & 5 & 6 & 7 & 8  & 9  & 10 & 11 & 12 & 13 & 14 & 15 & 16 &17&18&19&20\\
N_5(d)  & d & 6 & 7 & 9 & 10 & 12 & 16 & 18 & 20$--$21 & 26 & 28$--$29 & 36 & 40$--$41& 48$--$50& 48$--$61& 72$--$76 & 90$--$96  \\
\end{array}\]
\[\begin{array}{c|cccccccccccc}
d       & 21 & 22 & 23$--$60 & 61$--$136       & 137$--$185 & 186$--$\left\lceil(2V+2)/3\right\rceil & \left\lceil(2V+5)/3\right\rceil$--$\\
N_5(d)  & 126 & 176 & 276 & 276$--$B(d) & 276$--$d(d+1)/2 & \left\lfloor 3(d-1)/2\right\rfloor$--$V & \left\lfloor 3(d-1)/2\right\rfloor\\
\end{array}\]
}
\begingroup
\captionof{table}{Bounds for the maximum number of equiangular lines with $\alpha=1/5$.}\label{newT}
\endgroup
\begin{proof}
We compare upper bounds given by Proposition~\ref{s2p23} with lower bounds arising from direct constructions. For $d\leq 8$ we have $n\leq 12$ and we infer the result from the full classification of Seidel matrices (see Section \ref{sec5}). For $d=9$ we have $n\leq 13$, but $n=13$ lines could have only been obtained by extending one of the four Seidel matrices of order $12$ with smallest eigenvalue $-5$ of multiplicity $3$. This is shown to be impossible by a simple computer search. For $d=10$ there is equality in the relative bound and $16$ lines can be obtained from the symmetric Hadamard matrix $H=(J_4-2I_4)\otimes (J_4-2I_4)$ after removing its diagonal. The Seidel matrix $H-I_{16}$ can be extended, via an easy computer search, with further two lines to obtain $18$ equiangular lines in $\mathbb{R}^{11}$ (see Remark~\ref{remuselater}). For $d=13$ we have four inequivalent conference graphs \cite{MSP, TS}, and an application of Theorem~\ref{S4MC} shows the existence of $20$ lines in $\mathbb{R}^{12}$, while $22$ lines are impossible by Example~\ref{DRJACK}. Finally, cases $14\leq d\leq 23$ agree with the values of Table \ref{xyz}; cases $24\leq d\leq 136$ along with the upper bounds $B(d)$ are discussed in \cite{ABARG}; and the remaining values follow from Corollary~\ref{neumcor}.
\end{proof}
Finally, we remark that it would be nice to see a combinatorial interpretation of Seidel matrices with three distinct eigenvalues. Such new perspective might shed some light on the existence of the hypothetical Seidel matrices highlighted in the appendix. This will hopefully lead to further improvements upon the best known bounds on the number of equiangular lines in small dimensions.
\section*{Acknowledgments}
We are greatly indebted to Andries Brouwer and Donald Taylor for their helpful remarks regarding questions raised during the preparation of this manuscript. We are grateful for Alexander Barg and Wei-Hsuan Yu for sharing their notes prior to publication. We also thank the reviewers for their thorough work. This work was supported in part by the Hungarian National Research Fund OTKA K-77748; by the JSPS KAKENHI Grant Numbers $24\cdot02789$, and $24\cdot02807$; and by the ``$100$ talent'' program of the Chinese Academy of Sciences. Part of the computational results in this research were obtained using supercomputing resources at Cyberscience Center, Tohoku University. 

\appendix
\section{A supplementary table}\label{Ast}
Here we display a list of feasible spectra for Seidel matrices whose existence would lead to an attainment of or improvement upon the known maximum number of equiangular lines in dimensions $14$ and $16$--$20$. The table was compiled using Theorem~\ref{3g}, Corollary~\ref{S2mC}, and Corollary~\ref{thisstr}.
\[\begin{array}{cccccccc}
n & d & \lambda & \mu & \nu & \text{Existence} & \text{Remark}\\
\hline
 28 & 14 & \left[-5\right]^{14} & \left[3\right]^7 & \left[7\right]^7     & \text{Y} & \text{Example~\ref{ex415}}\\
 30 & 14 & \left[-5\right]^{16} & \left[5\right]^9 & \left[7\right]^5     & \text{N} & \text{Theorem~\ref{nonex}}\\
 40 & 16 & \left[-5\right]^{24} & \left[5\right]^6 & \left[9\right]^{10}  & ? &\\
 40 & 16 & \left[-5\right]^{24} & \left[7\right]^{15} & \left[15\right]^1 & \text{Y} & \text{Example~\ref{rank3g}}\\
 42 & 16 & \left[-5\right]^{26} & \left[7\right]^7 & \left[9\right]^9     & \text{N} & \text{Theorem~\ref{nonex}}\\
 48 & 17 & \left[-5\right]^{31} & \left[7\right]^8 & \left[11\right]^9    & \text{Y} & \text{Example~\ref{STS19}}\\
 49 & 17 & \left[-5\right]^{32} & \left[9\right]^{16} & \left[16\right]^1 & ? &\\ 
 48 & 18 & \left[-5\right]^{30} & \left[3\right]^6 & \left[11\right]^{12} & ? &\\
 48 & 18 & \left[-5\right]^{30} & \left[7\right]^{16} & \left[19\right]^2 & ? &\\
 54 & 18 & \left[-5\right]^{36} & \left[7\right]^9 & \left[13\right]^9    & ? &\\
 60 & 18 & \left[-5\right]^{42} & \left[11\right]^{15} & \left[15\right]^3& ? &\text{Remark~\ref{S4C2}} \\
 72 & 19 & \left[-5\right]^{53} & \left[13\right]^{16} & \left[19\right]^3& \text{Y} &\text{Example~\ref{ASCh}} \\
 75 & 19 & \left[-5\right]^{56} & \left[10\right]^1 & \left[15\right]^{18}& ? &\text{Lemma~\ref{S4L8}} \\
 90 & 20 & \left[-5\right]^{70} & \left[13\right]^5 & \left[19\right]^{15}& ? & \\
 95 & 20 & \left[-5\right]^{75} & \left[14\right]^1 & \left[19\right]^{19}& ? &\text{Lemma~\ref{S4L8}}\\
\end{array}\]
\begingroup
\captionof{table}{Feasible parameters of Seidel matrices with exactly three distinct eigenvalues.}
\endgroup
\end{document}